\documentclass[a4paper, 12pt]{amsart}

\usepackage{ucs}
\usepackage[utf8x]{inputenc}
\usepackage[ngerman, english]{babel}
\usepackage{amsfonts,amsmath, amssymb}
\usepackage{fancyhdr}
\usepackage[fixlanguage]{babelbib}
\usepackage{url}
\usepackage{hyperref}
\usepackage{color}
\usepackage{cite}
\usepackage{mathtools}

\theoremstyle{abstrace}

\newtheorem*{thm1}{Main Theorem}
\newtheorem{thm}{Theorem}[section]
\newtheorem{lem}[thm]{Lemma}
\newtheorem{prop}[thm]{Proposition}
\newtheorem{cor}[thm]{Corollary}

\theoremstyle{definition}

\theoremstyle{remark}
\newtheorem{rem}[thm]{Remark}

\newcommand{\Norm}[2]{\Vert #1 \Vert_{L^{#2}}}
\newcommand{\con}[1]{\nabla_{#1}}

\newcommand{\laplace}{\nabla^{\ast} \nabla}
\newcommand{\summei}{\sum_{i=1}^n}
\newcommand{\summeij}{\sum_{i,j=1}^n}
\newcommand{\eps}{\varepsilon}
\newcommand{\feps}{f_{\varepsilon}}
\newcommand{\onframe}{(e_1, \ldots , e_n)}
\newcommand{\smooth}{C^{\infty}(M)}

\DeclareMathOperator{\Ric}{Ric}
\DeclareMathOperator{\diam}{diam}
\DeclareMathOperator{\R}{R}
\DeclareMathOperator{\vol}{vol}

\begin{document}

\nocite{*}

\title{Estimates for eigensections of Riemannian vector bundles}

\author{Saskia Roos}
\address{Max-Planck-Institut für Mathematik, Vivatsgasse 7, 53111 Bonn, Germany}
\email{saroos@mpim-bonn.mpg.de}

\begin{abstract}
We derive a bound on the $L^{\infty}$-norm of the covariant derivative of Laplace eigensections on general Riemannian vector bundles depending on the diameter, the dimension, the Ricci curvature of the underlying manifold, and the curvature of the Riemannian vector bundle. Our result implies that eigensections with small eigenvalues are almost parallel.  
\end{abstract}

\maketitle

\section{Introduction}

There is a close relation between the topology of a closed Riemannian manifold and the spectrum of its Laplacian. For example, the Laplacian acting on 1-forms determines the first Betti number. By a classical theorem of Bochner the first Betti number of an $n$-dimensional connected compact Riemannian manifold with nonnegative Ricci curvature satisfies $b_1(M) \leq n$ with equality if and only if $M$ is isometric to a flat torus. It was proved in \cite{petersensprouse} that $n$-dimensional compact Riemannian manifolds with $n$ small Laplace eigenvalues are diffeomorphic to a nilmanifold. Ammann and Sprouse showed a similar result for Dirac eigenvalues on compact spin manifolds in \cite{ammann}. Both statements were consequences of upper bounds on the $L^{\infty}$-norm of $L^2$-eigensections and of their covariant derivatives on general Riemannian vector bundles. 
\newline

In this paper we refine the upper bound on the $L^{\infty}$-norm  of the covariant derivative of an $L^2$-eigensection. For this, we need the following upper bound on the $L^{\infty}$-norm of an eigensection which is proved by standard Moser iteration, see for instance \cite[Lemma 4]{ballmann}.

\begin{lem}\label{SInfinity}
Let $E$ be a Riemannian vector bundle over a closed Riemannian manifold of dimension $n \geq 3$, equipped with a metric connection $\nabla$ such that $\Ric_g \geq  -(n-1)K$ and $\diam(M) \leq d$ for some positive numbers $K$ and $d$. Assume that for an $L^2$-section $S$ there is a positive number $\lambda$ such that $\langle \laplace S, S \rangle \leq \lambda \vert S \vert^2$ holds pointwise. Then there is a positive constant $c = c(n, \sqrt{K}d)$ such that
\begin{align*}
\Norm{S}{\infty} \leq \exp \bigg(2 \sqrt{\lambda} c d  \frac{\sqrt{p}}{(2-p)(\sqrt{p} - \sqrt{2-p})} \bigg) \Norm{S}{2}
\end{align*}
for any $p \in \left(1,\frac{n}{n-1}\right]$.
\end{lem}

Note that the $L^p$-norms for $1 \leq p \leq \infty$ are volume-normalized (see Section 2 for a definiton). The constant $c(n,\sqrt{K}d)$ is explicitly given in the Sobolev estimate proven by Gallot (see Theorem \ref{Gallot}).
\newline

The proofs of \cite[Theorem 1.1]{petersensprouse} and \cite[Theorem 1.1]{ammann} use an upper bound on the $L^{\infty}$-norm of the covariant derivative of an eigensection $S$. Since the corresponding Bochner formula contains the covariant derivative of the curvature $\nabla \R^E$ of the Riemannian vector bundle, the bounds on $\Norm{\nabla S}{\infty}$ derived in \cite[Theorem 4.2]{petersensprouse} and \cite[Lemma A.4]{ammann} also depend on a uniform upper bound for $\vert \nabla \R^E \vert$. In the proofs of \cite[Theorem 1.1]{petersensprouse} and \cite[Theorem 1.1]{ammann} the dependency on $\nabla \R^E$ is removed by changing the metric according to Abresch's Theorem \cite{Abresch} to obtain a uniform upper bound on $\nabla \R$. However, in \cite{ammannerratum} Ammann points out that one can not Abresch's Theorem. This is because it could happen that the upper bound on the $r$-th eigenvalue for the Laplacian which is derived for the deformed metric is smaller than the perturbation of the Laplace eigenvalues due to the metric change. But by choosing a closer metric the bound on $\nabla R$ becomes larger. Since the bound on $\Norm{\nabla S}{\infty}$ used in \cite{petersensprouse} and \cite{ammann} depends on $\nabla R$ the upper bound for the $r$-th Laplace eigenvalue for this deformed metric becomes smaller. 

Thus, the goal of this paper is to give an upper bound on the $L^{\infty}$-norm of the covariant derivative of an eigensection that only depends on the dimension, the diameter, the lower bound on the Ricci curvature of the manifold, and the curvature of the vector bundle. 

\begin{thm1}
Let $(M,g)$ be a closed Riemannian manifold such that $\Ric_g \geq - (n-1)K$ and $\diam (M) \leq d$ for some positive constants $K$ and $d$. Furthermore, let $E$ be a Riemannian vector bundle over $M$ equipped with a metric connection $\nabla$ such that $\vert \R^E \vert \leq r$ for some positive constant $r$. Then any section $S$ with $\laplace S = \lambda S $ for some positive $\lambda$ satisfies the following inequality:
\begin{equation*}
\Norm{\nabla S}{\infty} \leq \lambda^{\alpha(n)} \exp\left(A(n,K,d) \sqrt{2 \lambda}\right) B(n,K,d,r) \Norm{S}{2},
\end{equation*}
where $\alpha(n) = \frac{1}{2}$ if $\lambda \geq 1$ and $\alpha(n)=  \frac{\eps(n)n}{2n+4(1-\eps(n))}$ else. The positive constants are explicitly given by:
\begin{align*}
\eps(n) & \coloneqq \prod_{i=1}^{\infty}\left(1 - \left( \frac{n}{n+2} \right) ^ i \ \right) ,
\end{align*}
\newpage
\begin{align*}
A(n,K,d) & \coloneqq \frac{(n+2)(n+1)}{n\eps(n)} c(n,\sqrt{K}d) d , \\[7pt]
B(n,K,d,r) & \coloneqq \left( \frac{n+2}{n} \right)^{ \frac{(n + 2)^2}{4n\eps(n)} } \exp \bigg( \frac{(n+2)(n+1)}{n\eps(n)}c(n, \sqrt{K}d) d \\
& \qquad \qquad \qquad  \qquad \qquad  \cdot \sqrt{2( (n-1)K + n^2(r + r^2))} \bigg),
\end{align*}
and $c(n, \sqrt{K}d )$ is the Sobolev constant explicitly given in Theorem \ref{Gallot}.
If in addition $\Norm{\nabla S}{\infty} \geq  \Norm{S}{\infty}$ then $\alpha(n) = \frac{1}{2}$. 
\end{thm1}

\begin{rem}
If one knows a priori that $\Norm{\nabla S}{\infty} < \Norm{S}{\infty}$, then the constants $A$ and $B$ can be improved. This will be discussed in detail in the proof of the main theorem.
\end{rem}

This estimate is analogous to \cite[Theorem 2.2]{Colbois} with vanishing potential. However, there is a difference to our result. In our main theorem the Ricci curvature can have any lower bound whereas in \cite[Theorem 2.2]{Colbois} the $L^{\frac{p}{2}}$-norm of the negative part of the Ricci curvature has to be smaller than a constant depending on $n$, on chosen real numbers $p > q > n$ and on the diameter of the manifold, to get a uniform bound on the Sobolev constant. Using these estimates Aubry, Colbois, Ghanaat, and Ruh improved the spectral characterization of nilmanifolds started by Petersen and Sprouse. Furthermore, Ammann referred in his small erratum \cite{ammannerratum} to this estimate to fix the proof of \cite[Theorem 1.1]{ammann}. Nonetheless, the assumption on the Ricci curvature changes in comparison to the original theorems. In Section 3 we show how our main theorem can be used to prove \cite[Theorem 1.1]{petersensprouse} and \cite[Theorem 1.1]{ammann} in their original statement.
\newline

For the proof of our main theorem we follow closely the proof of \cite[Theorem 2]{ballmann}. There, Ballmann, Br\"uning, and Carron derived a bound on $\Norm{\nabla S}{\infty}$ which is independent of $\nabla \R$ but depends on the holonomy of the vector bundle $E$. Their motivation was to study the relation between eigenvalues and holonomy. For the main theorem, however, the holonomy estimate is replaced by setting $N = \max \lbrace \Norm{\nabla S}{\infty} , \Norm{S}{\infty} \rbrace$ which leads to a case analysis. The advantage of our main theorem is that it makes the dependency of the upper bound for $\Norm{\nabla S}{\infty}$ independent of the holonomy.

After the proof we explain how this estimate influences the proofs of \cite[Theorem 1.1]{petersensprouse} and \cite[Theorem 1.1]{ammann}. Then we show that our main theorem leads to a similar result as \cite[Theorem 2]{ballmann} which only depends polynomially on the holonomy instead of exponentially.

\subsection*{Acknowledgment}
These studies started in my master thesis and it is a great pleasure for me to thank my supervisors Werner Ballmann and Bernd Ammann for suggesting this problem to me and many helpful discussions. In addition, I would like to extend my special thanks to Christian Blohmann, Fabian Spiegel and Asma Hassannezhad for their advisory support. Furthermore I would like to thank Erwann Aubry for helpful comments. Moreover I gratefully acknowledge the support and hospitality of the Max Planck Institute for Mathematics in Bonn.

\section{Proof of the Main Theorem}

We follow the proof of \cite[Theorem 2]{ballmann} and emphasize the change which is done to avoid the dependency on the holonomy.

Throughout this paper, we use the volume-normalized $L^p$-norms. For $E$ and $M$ as in the assumptions of the main theorem, we define the volume normalized $L^p$-norm to be
\begin{align*}
\Norm{S}{p} = \left( \frac{1}{\vol(M)} \int_M \vert S \vert^p  \right)^{\frac{1}{p}}
\end{align*}
for $L^p$-sections $S$ of $E$.

The main tool of this proof is Moser iteration. For this, we need the following Sobolev estimate proven by Gallot, \cite{gallot}. 

\begin{thm}\label{Gallot}
Let $(M,g)$ be a closed Riemannian manifold of dimension $n \geq 3$, with $\Ric_g \geq -(n-1)K$ and $\diam (M) \leq d$ for some positive constants $K, d$. Then there is a positive constant $c = c(n, \sqrt{K} d)$ such that for all $p \in \left[ 1 , \frac{n}{n-1} \right]$ and $f \in \smooth$ we have the inequality
\begin{align*}
\Norm{f}{\frac{2p}{2-p}} \leq \Norm{f}{2} + \frac{2c}{2-p} \diam(M) \Norm{df}{2}.
\end{align*}
The constant $c$ is given by 
\begin{align*}
c(n, D) = \frac{1}{D} \int_0^{D} \left(\frac{1}{2}e^{(n-1)D} \cosh(t) + \frac{1}{nD} \sinh(t)\right)^{(n-1)} \mathrm{d}t ,
\end{align*}
where $D \coloneqq \sqrt{K}d$.
\end{thm}

As we want to apply this theorem to obtain a Moser type iteration, we need to assume the Ricci and diameter bounds stated in the main theorem.

The strategy is to apply Theorem \ref{Gallot} to $\vert \nabla S \vert$. Since this is not a smooth function, we define $\feps \coloneqq \sqrt{\vert \nabla S \vert^2 + \eps^2}$ for small $\eps > 0$.

Using the identity $d \feps^k = k \feps^{k-1} d\feps$ we obtain for $k > \frac{1}{2}$
\begin{equation}\label{dfIdentity}
\begin{aligned}
\Norm{d\feps^k}{2}^2 & = k^2 \langle d\feps , \feps^{2k-2} d \feps \rangle_{L^2} \\
& = \frac{k^2}{2k-1} \langle d \feps , d\feps^{2k-1} \rangle_{L^2} \\
& = \frac{k^2}{2k-1} \langle \Delta \feps , \feps^{2k-1} \rangle_{L^2},
\end{aligned}
\end{equation}
where $\langle . , .\rangle_{L^2}$ denotes the induced volume-normalized $L^2$-inner product.

Let $\langle . , . \rangle$ denote the standard fibrewise inner product. Following the proof of \cite[Lemma A.4]{ammann} we have on the one hand
\begin{align*}
\Delta \feps^2 =2 \feps \Delta \feps - 2 \big\vert d \vert \feps \vert \big\vert^2.
\end{align*}
On the other hand we calculate
\begin{align*}
\Delta \feps^2 = \Delta(\vert \nabla S \vert^2 + \eps^2) &= 2\langle (\laplace ) \nabla S , \nabla S \rangle - 2 \vert \nabla \nabla S \vert^2 \\
& \leq 2\langle (\laplace ) \nabla S , \nabla S \rangle - 2 \big\vert d \vert \nabla S \vert \big\vert^2 ,
\end{align*}
where the last inequality follows from Kato's inequality (see for instance \cite{Kato}). As $\big\vert d \vert \feps \vert \big\vert \leq \big\vert d \vert \nabla S \vert \big\vert$ we obtain that
\begin{align*}
\feps \Delta \feps \leq \langle (\laplace) \nabla S , \nabla S \rangle.
\end{align*}

For a local vector field $Z$ and a local orthonormal frame $\onframe$ the Bochner formula yields
\begin{align*}
((\laplace) \nabla S) (Z) = & \con{Z}(\laplace S) - \con{\Ric(Z)}S -  2 \summei \R^E(e_i , Z) \con{e_i}S \\
&  - \summei (\con{e_i} \R^E)(e_i, Z)S.
\end{align*}
Using the assumptions on the manifold and the vector bundle we estimate
\begin{align*}
\feps \Delta \feps & \leq \langle (\laplace) \nabla S , \nabla S \rangle \\
 & \leq  (\lambda + (n-1)K + n^2 r) \vert \nabla S \vert^2 \\
& \ \ \ - \summeij \Big\langle \big( (\con{e_i} \R^E)(e_i,e_j) + \R^E(e_i, e_j)\con{e_i}S \big) ,\con{e_j}S \Big\rangle.
\end{align*}
Combining this with the identity \eqref{dfIdentity} we obtain
\begin{align*}
\Norm{d \feps^k}{2}^2 \leq & \  \frac{k^2}{2k-1} (\lambda + (n-1)K + n^2 r)) \Norm{\feps^k}{2}^2 \\
& - \frac{k^2}{2k-1} \ \, \frac{1}{\vol(M)} \int_M \summeij \langle \con{e_i} \big( \R^E(e_i, e_j)S \big), \con{e_j}S \rangle \feps^{2k-2}.
\end{align*}
For each $p \in M$ separately we choose $\onframe$ such that it is parallel at $p$. Using the divergence formula as in \cite[page 267]{divergenz}, we rewrite the integral as follows
\begin{equation}\label{integral}
\begin{aligned}
& -  \int_M \summeij \langle \con{e_i}\big(\R^E(e_i, e_j)S \big), \con{e_j}S \rangle \feps^{2k-2} \\
& \ \ \  \ \ \  = \int_M \summeij \langle \R^E(e_i, e_j)S, \con{e_i}\con{e_j} S \rangle \feps^{2k-2} \\
& \ \ \  \ \ \  \ \  \ +2(k-1)\int_M \feps^{2k-3} \summeij d \feps(e_i) \langle \R^E(e_i,e_j)S, \con{e_j}S \rangle.
\end{aligned}
\end{equation}
By our choice of the local orthonormal frame, we have pointwise
\begin{align*}
\summeij \langle \R^E(e_i, e_j)S, \con{e_i} \con{e_j}S \rangle = \frac{1}{2}\summeij \vert \R^E(e_i, e_j)S \vert^2.
\end{align*}
Therefore, we obtain
\begin{align*}
&  2(k-1)\int_M \feps^{2k-3} \summeij d \feps(e_i) \langle \R^E(e_i,e_j)S, \con{e_j}S \rangle.\\
& \ \ \ \ \ \ \leq \frac{n^2r^2}{2}\int_M \vert S \vert^2 \feps^{2k-2} + 2(k-1)nr \int_M \vert S \vert \feps^{2k-2} \vert d\feps \vert \\
& \ \ \ \ \ \ \leq \frac{n^2 r^2}{2} \Norm{S}{\infty}^2 \int_M \feps^{2k-2} + 2\frac{k-1}{k}nr \Norm{S}{\infty} \int_M \feps^{k-1}\vert d\feps^k \vert.
\end{align*}
We further estimate the second term of \eqref{integral}:
\begin{align*}
& 2\frac{k-1}{k}nr \Norm{S}{\infty} \int_M \feps^{k-1}\vert d\feps^k \vert \\
& \ \ \ \ \ \ \leq \frac{2k-1}{k^2} \Bigg[ \frac{1}{2} \int_M \vert d \feps^k \vert^2  + \left( \frac{k(k-1)}{2k-1} \right)^2 2 n^2 r^2 \Norm{S}{\infty}^2 \int_M \feps^{2k-2} \Bigg].
\end{align*}
Up to this point we just followed the proof of \cite[Theorem 2]{ballmann}. Now, instead of estimating $\Norm{S}{\infty}$ against $\Norm{\nabla S}{\infty}$ with the holonomy of the vector bundle we will just leave the last inequality as it is in order to avoid the dependency on the holonomy. This does not change much but will require a case analysis at the end.

Combining the above inequalities, we have for $k > \frac{1}{2}$
\begin{align*}
\Norm{d \feps^k}{2}^2 & \leq  \frac{k^2}{2k-1} (\lambda + (n -1)K+n^2r)\Norm{\feps^k}{2}^2  \\
& \ \ \ - \frac{k^2}{2k-1}\frac{1}{\vol(M)} \int_M \summeij \langle \con{e_i}\big( \R^E(e_i,e_j)S \big),\con{e_j}S\rangle \feps^{2k-2}\\
& \leq  \frac{k^2}{2k-1} \bigg[ (\lambda + (n-1)K +n^2r) \Norm{\feps^k}{2}^2 + \frac{n^2r^2}{2}\Norm{S}{\infty}\Norm{\feps^{k-1}}{2}^2 \bigg] \\
& \  \ \  + \frac{1}{2} \Norm{d\feps^k}{2}^2 + \left( \frac{k(k-1)}{2k-1}\right)^2 2n^2r^2 \Norm{S}{\infty}^2 \Norm{\feps^{k-1}}{2}^2 \\
& \leq  k^2 \big(\lambda + (n-1)K+n^2r \big) \Norm{\feps^k}{2}^2 \\
& \ \ \ + k^2n^2r^2\left( \frac{1}{2} + \frac{2(k-1)^2}{(2k-1)^2} \right) \Norm{S}{\infty}\Norm{\feps^{k-1}}{2}^2 + \frac{1}{2}\Norm{d \feps^k}{2}^2.
\end{align*}
Since
\begin{align*}
\frac{2(k-1)^2}{(2k-1)^2} \leq \frac{1}{2}
\end{align*}
for $k \geq \frac{3}{4}$, we arrive at 
\begin{align*}
\Norm{d\feps^k}{2}^2 \leq 2k^2(\lambda + (n-1)K+n^2r)\Norm{\feps^k}{2}^2 + 2k^2n^2r^2\Norm{S}{\infty}^2\Norm{\feps^{k-1}}{2}^2.
\end{align*}
Using Hölder's inequality we find
\begin{align*}
\Norm{\feps^{k-1}}{2}^2 \leq \Norm{\feps}{2k}^{2k-2},
\end{align*}
for the volume-normalized $L^p$-norms. 

Set $G^2 \coloneqq 2(\lambda +(n-1)K+n^2(r+r^2))$ and $N_{\eps} \coloneqq \max\lbrace\Norm{\feps}{\infty} , \Norm{S}{\infty} \rbrace$. Then,
\begin{align*}
\Norm{d \feps^k}{2}^2 \leq k^2 G^2 N_{\eps}^2 \Norm{\feps}{2k}^{2k-2}.
\end{align*}
Taking the square root on both sides yields
\begin{align*}
\Norm{d \feps^k}{2} \leq k G N_{\eps} \Norm{\feps}{2k}^{k-1}.
\end{align*}

Now we apply Theorem \ref{Gallot} with $p = \frac{n+1}{n+2}$, thus $q = \frac{n+2}{n}$, and find a constant $C = C(n, \sqrt{K}d) = \frac{2n+2}{n} c(n, \sqrt{K}d) d$ such that
\begin{align*}
\Norm{\feps}{2kq}^k = \Norm{\feps^k}{2q} & \leq \Norm{\feps^k}{2} + C \Norm{d \feps^k}{2} \\
& \leq (1+ kCG) N_{\eps} \Norm{\feps}{2k}^{k-1}.
\end{align*}
Letting $\eps$ go to $0$ we obtain
\begin{equation}\label{Iteration}
\Norm{\nabla S}{2kq} \leq \left( 1 + kCG\right)^{\frac{1}{k}} N^{\frac{1}{k}}  \Norm{\nabla S}{2k}^{\frac{k-1}{k}},
\end{equation}
where $N \coloneqq \max\lbrace\Norm{\nabla S}{\infty} , \Norm{S}{\infty} \rbrace$.

Set $k \coloneqq q^j$ for $j = 1, 2, \ldots$ and $p_j \coloneqq 1 - \frac{1}{q^j}$. Then iterating the last inequality leads to
\begin{align*}
\Norm{\nabla S}{2q^{j+1}} & \leq (1 + q^j CG)^{\frac{1}{q^j}} N^{1- p_j} \Norm{\nabla S}{2q^j}^{p_j} \\
& \leq \prod_{i=1}^j(1+q^iCG)^{\frac{p_{i+1} \cdots p_j}{q^i}} N^{1-p_j \cdots p_i} \Norm{\nabla S}{2q}^{p_1 \cdots p_j} \\
& \leq \prod_{i=1}^j(1 + q^iCG)^{\frac{1}{q^i}} N^{1-p_j \cdots p_i} \Norm{\nabla S}{2q}^{p_1 \cdots p_j}.
\end{align*}
In the last step we used that $x^p \leq x$ for $x \geq 1$ and $p \in (0,1]$. We stop the iteration process at this point, since otherwise there would be no dependency on $\Norm{\nabla S}{2}$ anymore as inequality \eqref{Iteration} with $k=1$ shows.

Now we want to consider the limit $j \rightarrow \infty$. For this we need to check, whether the product on the right hand side converge. Using the geometric series, we first note that the limit
\begin{align*}
\eps = \eps(n) \coloneqq \prod_{i=1}^{\infty} p_i
\end{align*}
exists and lies in $\left( e^{-\frac{n^2 + 2n}{4}}, e^{-\frac{n}{2}} \right)$. Next, we estimate
\begin{align*}
 \prod_{i=1}^{\infty}(1 + q^iCG)^{\frac{1}{q^i}} & \leq (1 + CG)^{\sum_{i=1}^{\infty} \frac{1}{q^i}} \ q^{\sum_{i=1}^{\infty} \frac{i}{q^i}} \\
& \leq  \exp \left( C G \sum_{i=1}^{\infty} \frac{1}{q^i} \right) \ q^{\sum_{i=1}^{\infty} \frac{i}{q^i}} \coloneqq P.
\end{align*}
Recall that $q = \frac{n+2}{n}$. A small computation yields
\begin{align*}
P = \Bigg[ \exp\left(\frac{n}{2} CG\right) \Bigg] \left( \frac{n+2}{n} \right)^{ \frac{n^2 + 2n}{4} } < \infty .
\end{align*}
Hence, 
\begin{equation}\label{inequalityA}
\Norm{\nabla S}{\infty} \leq P N^{1- \eps} \Norm{\nabla S}{2q}^{\eps}.
\end{equation}

Since $N = \max \lbrace \Norm{\nabla S}{\infty} , \Norm{S}{\infty} \rbrace$ we have to distinguish two cases. First assume that $N = \Norm{\nabla S}{\infty}$. Then inequality \eqref{inequalityA} reads 
\begin{align*}
\Norm{\nabla S}{\infty} \leq P \Norm{\nabla S}{\infty}^{1- \eps} \Norm{\nabla S}{2q}^{\eps}
\end{align*}
and we can continue as in the proof of Theorem 2 in \cite{ballmann}. The latter inequality is equivalent to
\begin{align*}
\Norm{\nabla S}{\infty} \leq P^{\frac{1}{\eps}} \Norm{\nabla S}{2q}.
\end{align*}
Since $q = \frac{n+2}{n}$ H\"older's inequality leads to
\begin{equation}\label{Holder}
\Norm{\nabla S}{2\frac{n+2}{n}} \leq \Norm{\nabla S}{2}^{\frac{n}{n+2}} \Norm{\nabla S}{\infty}^{\frac{2}{n+2}}
\end{equation}
Using this and that $\laplace S = \lambda S$, we obtain
\begin{align*}
\Norm{\nabla S}{\infty} \leq P^{\frac{n+2}{n\eps}} \sqrt{\lambda} \Norm{S}{2}.
\end{align*}
The constant $P$ still depends on $\lambda$. In order to determine this dependency and to compare the bounds of the different cases we calculate:
\begin{align*}
P^{\frac{n+2}{n\eps}} &= \left( \frac{n+2}{n} \right)^{ \frac{(n + 2)^2}{4n\eps} } \exp\left(\frac{n+2}{2\eps} C G \right)\\
& = a(n) \exp \Bigg( \frac{(n+2)(n+1)}{n\eps}c(n, \sqrt{K}d) d \\ 
& \qquad \qquad \quad \quad  \cdot \sqrt{2(\lambda + (n-1)K + n^2r + n^2r^2)} \Bigg) \\
&\leq a(n) \exp\big(b(n,K,d,r) \big) \exp\left(\frac{(n+2)(n+1)}{n\eps} c(n,\sqrt{K}d) d \sqrt{2\lambda} \right) \\
&\eqqcolon B_1(n,K,d,r) \exp\left( A_1(n,K,d) \sqrt{2 \lambda} \right),
\end{align*}
where we define
\begin{align*}
a(n) & \coloneqq \left( \frac{n+2}{n} \right)^{ \frac{(n + 2)^2}{4n\eps} } , \\[7pt]
b(n,K,d,r) & \coloneqq \frac{(n+2)(n+1)}{n\eps}c(n, \sqrt{K}d) d \\
& \quad \ \cdot \sqrt{2( (n-1)K + n^2(r + r^2))} ,\\[7pt]
B_1(n,K,d,r) & \coloneqq a(n) \exp(b(n,K,d,r)),\\[7pt]
A_1(n,K,d) & \coloneqq \frac{(n+2)(n+1)}{n\eps} c(n,\sqrt{K}d) d.
\end{align*}
Thus in the case $N = \Norm{\nabla S}{\infty}$ we finally arrive  at the inequality
\begin{align*}
\Norm{\nabla S}{\infty}\leq \sqrt{\lambda} \, \exp\left(A_1(n,K,d)\sqrt{2\lambda} \right)B_1(n,K,d,r)\Norm{S}{2}.
\end{align*}

We now turn to the second case, assuming that $N = \Norm{S}{\infty}$. Then inequality \eqref{inequalityA} reads 
\begin{align*}
\Norm{\nabla S}{\infty} \leq P \Norm{S}{\infty}^{1- \eps} \, \Norm{\nabla S}{2q}^{\eps}.
\end{align*}
Using inequality \eqref{Holder}, this is equivalent to
\begin{align*}
\Norm{\nabla S}{\infty} \leq P^{\frac{n+2}{n+2(1-\eps)}} \Norm{S}{\infty}^{\frac{(1-\eps)(n+2)}{n+2(1-\eps)}}\Norm{\nabla S}{2}^{\frac{\eps n}{n+2(1-\eps))}}.
\end{align*}
As $\laplace S = \lambda S$ we obtain
\begin{equation}\label{inequalityRaw}
\Norm{\nabla S}{\infty} \leq P^{\frac{n+2}{n+2(1-\eps)}} \lambda^{\frac{\eps n}{2n+4(1-\eps)}} \Norm{S}{2}^{\frac{\eps n}{n+2(1-\eps))}} \Norm{S}{\infty}^{\frac{(1-\eps)(n+2)}{n+2(1-\eps)}}.
\end{equation}
Applying Lemma \ref{SInfinity} with $p= \frac{n+2}{n+1}$, we find
\begin{equation}\label{inequalitySInfinite}
\Norm{S}{\infty} \leq \exp \left( \sqrt{2} c(n, \sqrt{K}d) d  \frac{(n+1)\sqrt{n+2}}{n(\sqrt{n+2} - \sqrt{n})} \,  \sqrt{2\lambda} \right) \Norm{S}{2}.
\end{equation}
For an explicit expression for $P$ in \eqref{inequalityRaw} we obtain, using the same notation as before,
\begin{equation}\label{explicitP2}
\begin{aligned}
P^{\frac{n+2}{n+2(1-\eps)}} & \leq a(n)^{\frac{\eps n}{n+ 2(1-\eps)}}\exp\left(\frac{\eps n}{n+2(1-\eps)} b(n,K,d,r)\right)\\
& \ \ \ \exp\left(\frac{(n+2)(n+1)}{n+2(1-\eps)} c(n,\sqrt{K}d)d \, \sqrt{2 \lambda} \right) \\
& \eqqcolon B_2(n,K,d,r) \exp\left(\frac{(n+2)(n+1)}{n+2(1-\eps)} c(n,\sqrt{K}d)d \, \sqrt{2 \lambda} \right),
\end{aligned}
\end{equation}
where we set
\begin{align*}
B_2 \coloneqq a(n)^{\frac{\eps n}{n+ 2(1-\eps)}}\exp\left(\frac{\eps n}{n+2(1-\eps)} b(n,K,d,r)\right).
\end{align*}
However we still need the analogous calculation for the constant $A_1$. We insert \eqref{inequalitySInfinite} and \eqref{explicitP2} in  \eqref{inequalityRaw}:
\begin{align*}
\Norm{\nabla S}{\infty} & \leq  B_2(n,K,d,r) \exp\left(\frac{(n+2)(n+1)}{n+2(1-\eps)} c(n,\sqrt{K}d)d \, \sqrt{2 \lambda} \right) \\
& \ \ \, \cdot \exp \left( \frac{(1-\eps)(n+2)}{n+2(1-\eps)} \sqrt{2} c(n, \sqrt{K}d) d  \frac{(n+1)\sqrt{n+2}}{n(\sqrt{n+2} - \sqrt{n})} \,  \sqrt{2\lambda} \right)  \\
& \ \ \, \cdot \lambda^{\frac{\eps n}{2n+4(1-\eps)}} \Norm{S}{2}^{\frac{(1-\eps)(n+2)}{n+2(1-\eps)}} \Norm{S}{2}^{\frac{\eps n}{n+2(1-\eps))}}.
\end{align*}
Considering the combination of the exponential factors containing $\sqrt{\lambda}$, we obtain
\begin{align*}
& \ \ \ \ \exp\Bigg[ c(n,\sqrt{K}d)d  \, \frac{(n+2)(n+1)}{n+2(1-\eps)}\left( 1 + \frac{\sqrt{2}(1-\eps)\sqrt{n+2}}{n(\sqrt{n+2}-\sqrt{n})} \right) \sqrt{2\lambda} \, \Bigg]\\
&\eqqcolon \exp\left(A_2(n,K,d) \sqrt{2 \lambda} \right).
\end{align*}
Hence, we arrive at
\begin{align*}
\Norm{\nabla S}{\infty} \leq \lambda^{\frac{n \eps}{2n + 4(1-\eps)}} \exp\left( A_2(n,K,d) \sqrt{2 \lambda}\right) B_2(n,K,d,r) \Norm{S}{2}.
\end{align*}

For an estimate that holds in both cases, we need to compare $A_1$ with $A_2$ and $B_1$ with $B_2$. Comparing $B_1$ with $B_2$ we note that
\begin{align*}
B_1 = B_2^{\frac{n+2(1-\eps)}{\eps n}}.
\end{align*}
As $B_2$ and $\frac{n+2(1-\eps)}{\eps n}$ are both larger than $1$ we conclude that $B_1 > B_2$. Straight-forward calculations and using that $\eps \in \left( e^{-\frac{n^2 + 2n}{4}}, e^{-\frac{n}{2}} \right)$ we conclude $A_1 > A_2$. It remains to compare the $\lambda$-factors, for which we only need to observe that
\begin{align*}
\lambda^{\frac{\eps n}{2n + 4(1-\eps)}} \leq \lambda^{\frac{1}{2}}
\end{align*}
iff $\lambda \geq 1$. Otherwise, the inequality goes in the other direction. 

Finally, set $A \coloneqq A_1$, $B \coloneqq B_1$ and define $\alpha(n)$ to be $\frac{1}{2}$ if $\lambda \geq 1$ and to be $\frac{\eps n}{2n + 4(1-\eps)}$ if $\lambda < 1$.

\section{Applications}

As indicated in the introduction our main theorem has various applications. 
The first application is that a family of $L^2$-normalized eigensections on a Riemannian vector bundle is an almost pointwise orthonormal family if the corresponding eigenvalues are sufficiently small. To the author's knowledge this was first proved in \cite[Theorem 4.3]{petersensprouse} with a dependency on the divergence of the curvature. Another proof of this result can be found in \cite[Theorem 2.2]{ammann}.

In order to ease notation, we denote with $\tau (\eps \mid x_1, \ldots , x_n)$ a positive continuous function in all its variables which vanishes as $\eps$ does so while the other variables are fixed.

Using our main theorem together with Lemma \ref{SInfinity}, one can apply the same proof as in \cite[Theorem 4.3]{petersensprouse} or \cite[Theorem 2.2]{ammann} to derive the following:

\begin{prop}\label{nearlyON}
Let $E$ be a Riemannian vector bundle over a Riemannian manifold $M$ as in the main theorem equipped with a metric connection $\nabla$ such that $\vert \R^E \vert \leq r$ . Let $\lambda_i$ denote the $i$-th eigenvalue of $\laplace$ counted with multiplicities. Let $S_i$ and $S_j$ be $L^2$-normalized eigensections with eigenvalue $\lambda_i$, $\lambda_j$ respectively . Then
\begin{align*}
\Norm{\langle S_i , S_j \rangle - \delta_{ij}}{\infty} \leq \tau(\max \lbrace \lambda_i, \lambda_j \rbrace \mid n,K,d, r).
\end{align*}
\end{prop}

\begin{rem}
A similar statement was also used in \cite[Lemma 3.5]{Colding} to derive an upper bound on the Gromov-Hausdorff distant between balls in the euclidean space and balls on manifolds with a lower Ricci bound and bounded diameter.
\end{rem}

Assume now that $E$ has rank $k$. Then this proposition and the Gram-Schmidt process immediately leads to the following:
\begin{cor}\label{EigenvalueONframe}
There is a positive $\eps = \eps(n,K,d,r, k)$ such that $\lambda_k \leq \eps$ implies that $E$ is trivial. Furthermore, if $S_1, \ldots, S_k$ denote the first $k$ $L^2$-normalized eigensections, than we find a global orthonormal frame $\lbrace e_1, \ldots , e_k \rbrace$ such that
\begin{align*}
\Norm{e_i - S_i}{\infty} \leq \tau( \eps \mid n,K,d, r)
\end{align*}
for all $1 \leq i \leq k$. In particular, $\Norm{\nabla e_i}{\infty} \leq \tau( \eps \mid n,K,d, r)$.
\end{cor}
This corollary and the above proposition also imply directly that there is a positive lower bound $c_1(n,K,d,r,k)$ on the $(k+1)$-th eigenvalue.

Now, as outlined in the introduction, \cite{petersensprouse} and \cite{ammann} used similar estimates as in the main theorem and Proposition \ref{nearlyON}, with bounds depending on $\vert \nabla \R \vert$ to prove that enough small eigenvalues imply that a Riemannian resp.\ spin manifold is diffeomorphic to a nilmanifold. Both applied Abresch's Theorem (see \cite{Abresch} or Theorem 1.12 in \cite{CheegerFukuyaGromov}) to vary the metric such that $\vert \nabla \R \vert$ is uniformly bounded. However, as it is pointed out in the small erratum \cite{ammannerratum}, this does not work as $\vert \nabla \R \vert$ might grow when the change of the metric gets small. Using the bound given by our main theorem we avoid this problem.

As for Theorem 1.1 in \cite{ammann} we can apply the same proof by just exchanging the bound on $\Norm{\nabla S}{\infty}$ for an $L^2$-eigensection $S$ with the bound given in our main theorem and replacing \cite[Theorem 2.2]{ammann} by Proposition \ref{nearlyON}.

Petersen and Sprouse used an upper bound on $\Norm{\nabla \nabla S}{2}$ for an $L^2$-eigensection $S$ depending on $\vert \nabla \R \vert$ to prove Theorem 1.1 in \cite{petersensprouse}. As we were unable to bound $\Norm{\nabla \nabla S}{2}$ appropriately without assuming an upper bound on $\vert \nabla \R \vert$ we give another proof \cite[Theorem 1.1]{petersensprouse} using the same strategy as in \cite[Theorem 1.1]{ammann}.

\begin{thm}
Let $(M,g)$ be a closed $n$-dimensional Riemannian manifold such that $\diam(M) \leq d$ and $\vert \sec \vert \leq  K$. Then there is a positive $\eps = \eps(n, K, d)$ such that $\lambda_n (\laplace) \leq \eps$ implies that $M$ is a nilmanifold.
\end{thm}

\begin{proof}
By \cite[Theorem 1.2]{Ghanaat} there is a positive $\tilde{\eps}(n)$ such that a global orthonormal frame $\onframe$ on $TM$ with $\Norm{\nabla e_i}{\infty} \leq \tilde{\eps}(n)$ implies that $M$ is diffeomorphic to a nilmanifold. 

Considering Corollary \ref{EigenvalueONframe} we see that there is a positive $\eps = \eps(n,K,d)$ such that $\lambda_n \leq \eps$ implies that there is such  a global orthonormal frame $\onframe$ on $TM$ satisfying  $\Norm{\nabla e_i}{\infty} \leq \tilde{\eps}(n)$. Thus, the claim follows.
\end{proof}

As a further application, our main theorem leads to a version of \cite[Theorem 2]{ballmann} with different constants. This will give us a lower bound on the spectrum which in addition depends on the holonomy of the Riemannian vector bundle. In contrast to the lower bound given in \cite[Theorem 2]{ballmann}, the dependence on the holonomy here is only polynomial For this, we repeat the following definition from \cite[page 658]{ballmann} for a Riemannian vector bundle $E$ over a compact manifold $M$: Let $c$ be a non-constant loop in $M$. $L(c)$ denotes the length of $c$ and $H_c(.)$ the holonomy along it. Then for each $p \in M$ and unit vector $v \in E_x$ we define $\beta(v)$ to be the supremum of $\frac{\vert H_c(v) - v \vert}{L(c)}$ taken over all non-constant loops $c$ starting in $p$. Setting $\beta \coloneqq \inf \lbrace \beta(v) \vert v \in E, \vert v \vert = 1 \rbrace$ we state the following bound, which is also valid for $\beta = 0$.

\begin{thm}
Assume the same situation as in the main theorem. Then there are positive constants $a= a(n)$ and $c_1 = c_1(n \sqrt{K}d)$ such that any eigenvalue $\lambda$ of $\laplace$ satisfies
\begin{align*}
\lambda \geq \min \Bigg\lbrace 1, \, \beta^{\frac{2n+4(1-\eps)}{n\eps}}  a \exp \Big( - c_1 d \Big( 1+ \sqrt{(n-1)K + n^2(r + r^2)}  \Big)\Bigg\rbrace.
\end{align*}
The constants $a=a(n)$, $\eps = \eps(n)$ and $c_1=c_1(n,\sqrt{K}d)$ are determined explicitly.
\end{thm}

\begin{proof}

Ballmann, Br\"uning and Carron showed in the proof of \cite[Theorem 5]{ballmann} that
\begin{equation}\label{Holonomy}
\beta \Norm{S}{2} \leq \beta \Norm{S}{\infty} \leq \Norm{\nabla S}{\infty}.
\end{equation}
for any nontrivial $L^2$-section $S$.

Now let $S$ be an $L^2$-eigensection with eigenvalue $\lambda$. Combining \eqref{Holonomy} with our main theorem, we obtain
\begin{align*}
\beta \leq \lambda^{\alpha(n)} \exp\left(A(n,K,d,r) \sqrt{2 \lambda}\right) B(n,K,d,r),
\end{align*}
where $\alpha(n) = \frac{1}{2}$ if $\lambda \geq 1$ and $\alpha(n) = \frac{\eps n}{2n + 4(1-\eps)}$ otherwise  with $\eps$ as in the proof of the main theorem. Thus
\begin{align*}
\lambda \geq \left( \frac{\beta}{B(n,K,d,r)} \right)^\frac{1}{\alpha(n)} \exp \left(- \frac{A(n,K,d,r)}{\alpha(n)} \sqrt{2\lambda} \right)
\end{align*}

Assuming $\lambda < 1$ we have $\alpha(n) = \frac{\eps n}{2n + 4(1-\eps)}$. Writing out the constants explicitly the statement follows.
\end{proof}

\bibliography{literature.bib}
\bibliographystyle{amsalpha}
\end{document}